\documentclass{amsart}
\usepackage{amsmath,amssymb,latexsym,mathrsfs}
\usepackage [mathscr]{eucal}
\usepackage {enumerate}
\usepackage{color}

\newtheorem{thm}[subsection]{Theorem}

\newtheorem{prop}[subsection]{Proposition}
\newtheorem{cor}[subsection]{Corollary}

\theoremstyle{remark}
\newtheorem{defn}[subsection]{Definition}
\newtheorem{ex}[subsection]{Example}

\begin{document}

\title{Oscillation Revisited}

\author {Gerald Beer}
\address {Department of Mathematics, California State University Los Angeles,
5151 State University Drive, Los Angeles, California 90032, USA}
\email{gbeer@cslanet.calstatela.edu}

\author {Jiling Cao}
\address{School of Engineering, Computer and Mathematical Sciences,
	Auckland University of Technology, Private Bag 92006, Auckland 1142, New Zealand}
\email {jiling.cao@aut.ac.nz}

\dedicatory {D\'{e}di\'{e} \`{a}  Michel Th\'{e}ra pour son
soixante-dixi\`{e}me anniversaire}

 \thanks{The first author thanks the Auckland University of Technology for its
 hospitality in February 2016. The second author thanks the support
 of the National Natural Science Foundation of China, grant No. 11571158, and the
 paper was partially written when he visited Minnan Normal University in April 2016
 as Min Jiang Scholar Guest Professor.}

 \date {August 9, 2016} 

 \begin {abstract}
 In previous work by Beer and Levi \cite {BL1,BL2}, the authors studied the oscillation
 $\Omega (f,A)$ of a function $f$ between metric spaces $\langle X,d \rangle$ and
 $\langle Y,\rho \rangle$ at a nonempty subset $A$ of $X$, defined so that when
 $A =\{x\}$, we get $\Omega (f,\{x\}) = \omega (f,x)$, where $\omega (f,x)$ denotes
 the classical notion of oscillation of $f$ at the point $x \in X$.  The main purpose
 of this article is to formulate a general joint continuity result for $(f,A) \mapsto
 \Omega (f,A)$ valid for continuous functions.
 \end{abstract}

 \maketitle

 \section{Introduction} \label{sec: Intro}

 Let $\langle X,d \rangle$ and $\langle Y,\rho \rangle$ be two metric spaces, each
 with at least two points, and let  $S_d(x,\varepsilon)$ denote the open $d$-ball of
 radius $\varepsilon$ about $x \in X$.  Suppose that $f$ is a function from $X$ to
 $Y$ and $x \in X$ is arbitrary. Put
 \[
 \omega_n(f,x) := \textrm{diam}_\rho \ f\left(S_d\left(x,\frac{1}{n}\right)\right)
 \ \ \ (n \in \mathbb{N}),
 \]
 noting that the diameter of the image of the ball in the target space $Y$ could
 be infinite.  In any case, for each positive integer $n\in \mathbb N$,
 \[
 \omega_n(f,x) \geq \omega_{n+1}(f,x),
 \]
 so that
 \[
 \textrm{lim}_{n\to\infty}\ \omega_n(f,x) = \textrm{inf}_{n \in \mathbb{N}}\
 \omega_n(f,x)
 \]
 is an extended nonnegative real number that is called the \textit{oscillation} of
 $f$ at $x$ and is denoted by $\omega (f,x)$ (see, e.g., \cite [p. 78] {HS}). Some
 basic well-known facts about oscillation are the following:
 \begin{enumerate}
 \item[(1.1)] $\omega (f,x) = 0$ if and only if $f$ is continuous at $x$;
 \item[(1.2)] $x \mapsto \omega (f,x)$ is upper semicontinuous;
 \item[(1.3)] $f$ is globally uniformly continuous if and only if $\langle \omega_n
 (f,\cdot) \rangle$ converges uniformly to the zero function on $X$.
 \end{enumerate}
 %\textcolor{red}{For details on these facts, refer to \cite{?}.}
 %\vspace {5 mm}

 An oscillation function $\omega (f,\cdot)$ is thus a nonnegative extended real-valued upper semicontinuous function that must take on the value zero at each isolated point of the space, as each function is automatically continuous at isolated points.  Conversely, a function $g$ with these properties can be shown to be an isolation function for some Borel real-valued function defined on the space, as shown only fairly recently by Ewert and Ponomarev \cite [Theorem 4] {EP}.

 One way to define the oscillation of a function $f$ from $X$ to $Y$ at a nonempty
 subset $A$ of $X$ was proposed by Beer and Levi \cite {BL1,BL2}. Consistent with
 our notation for open balls, put
 \[
 S_d(A,\varepsilon) := \bigcup_{a \in A} S_d(a,\varepsilon);
 \]
 this union is often called the $\varepsilon$-\textit{enlargement} of the set $A$,
 refer to \cite {Be}.  Then for each $n \in \mathbb{N}$, we put
 \[
 \Omega_n(f,A) := \sup \ \left\{\rho (f(x),f(w)) : x, w \in S_d
 \left(A,\frac{1}{n}\right)
 \ \textrm{and} \ d(x,w) < \frac{1}{n}\right\},
 \]
 and call
 \[
 \Omega (f,A) := \lim_{n\to\infty} \ \Omega_n(f,A) = \inf_{n \in \mathbb{N}} \
 \Omega_n(f,A)
 \]
 the \textit{oscillation of} $f$ \textit{at} $A$.  Easily, if $A = \{x\}$, then
 $\Omega (f, \{x\}) = \omega(f,x)$.

 Concerning the oscillation $\Omega(f, A)$, one may ask the following question:
 What are the counterparts of properties (1.1) - (1.3)? First of all, note that
 neither continuity of $f$ on $X$ nor uniform continuity of $f$ restricted to $A$
 ensures that $\Omega (f,A)$ is zero or even finite: consider $f : [0,\infty)
 \times [0,\infty) \rightarrow \mathbb{R}$ defined by $f(x,y) = xy$, where
 \[
 A = \{(x,y) : x = 0 \ \textrm{or} \ y = 0\}.
 \]
 However, if $A$ is compact and $f$
 is globally continuous, then the standard proof of the uniform continuity of
 $f$ restricted to $A$ shows that $\Omega (f,A) = 0$.  More precisely, each
 nonempty subset $A$ on which each globally continuous function on $X$ has
 oscillation zero at $A$ has this characteristic property \cite [Theorem 5.2] {BL1}:
 each sequence  $\langle a_n \rangle$ in $A$ along which $\lim_{n\rightarrow\infty}
 \ d(a_n, X - \{a_n\}) = 0$ must cluster.   A subset that exhibits this property
 is called a \textit{UC-subset}; trivially, each relatively compact subset is a
 UC-subset. If $X$ is a UC-subset of itself, then the metric space is called a
 \textit{UC-space}; their characteristic properties were first systematically
 described by Atsuji \cite {At} (see also \cite {Be,JK}).

 The family of all nonempty UC-subsets, like the family of all nonempty
 relatively compact subsets, form a \textit{bornology}:
 \begin{enumerate}
 \item[(1.4)] they are an hereditary family;
 \item[(1.5)] they are stable under finite unions;
 \item[(1.6)] they form a cover of $X$.
 \end{enumerate}
 The largest bornology on $X$ is the family of all nonempty subsets $\mathscr
 {P}_0(X)$ and the smallest is the family of all nonempty finite subsets
 $\mathscr {F}_0(X)$. Three other bornologies of note are the family of all
 nonempty metrically bounded subsets, the family of all nonempty totally bounded
 subsets, and the family of all nonempty Bourbaki bounded subsets \cite {BG, Bo,
 GM, He, Vr}, also called the \emph{finitely chainable subsets} \cite {At}.

 Beer and Levi called $f$ \textit{strongly uniformly continuous} on $A$ provided
 $\Omega (f,A) = 0$ as this property obviously implies that $f$ restricted to $A$
 is uniformly continuous.  They characterized strong uniform continuity in various
 ways, most notably, in terms of the preservation of nearness to subsets of $A$
 \cite [Theorem 3.1] {BL1}, and in terms of the continuity of the induced direct
 image map from $\mathscr {P}_0(X)$ to $\mathscr {P}_0(Y)$ at points of $\mathscr
 {P}_0(A)$, where subsets of the domain and codomain are equipped with the the
 Hausdorff pseudometric topologies as determined by $d$ and $\rho$, respectively
 \cite [Theorem 3.3] {BL1}.  Strong uniform continuity of $f$ at $A$ is a
 variational alternative to the uniform continuity of the restriction of $f$ to
 $A$: for every $\varepsilon > 0$, there exists $\delta > 0$ such that for each
 $a \in A$ and $x \in X, d(a,x) < \delta$ implies $\rho (f(a),f(x)) < \varepsilon$
 \cite[Theorem 4.3] {BL1}.  Furthermore, $f$ is strongly uniformly continuous on $A$
 if and only if $\langle \omega_n (f,\cdot) \rangle$ converges uniformly to the
 zero function on $A$ \cite [Theorem 3.1] {BL1}, which presents rather convincing
 evidence that strong uniform continuity on a subset is the correct generalization
 of global uniform continuity.

 With respect to the Hausdorff extended pseudometric topology $\tau_{H_d}$
 determined by the Hausdorff distance $H_d$ on $\mathscr {P}_0(X)$ \cite
 [Definition 4.1.5] {KT}, the map $A \mapsto \Omega (f,A)$ is upper semicontinuous
 for an arbitrary function $f$ from $X$ to $Y$ \cite [Theorem 4.3] {BL1} which
 yields the known upper semicontinuity of $x \mapsto \omega (f,x)$ as a corollary,
 since $x \mapsto \{x\}$ is an isometric embedding of $X$ into the hyperspace.
 Obviously, one cannot expect continuity of $A \mapsto \Omega (f,A)$ for an
 arbitrary function $f$ with respect to any topology on $\mathscr {P}_0(X)$
 with respect to which $x \mapsto \{x\}$ is a topological embedding.  It is not
 even true that $A \mapsto \Omega (f,A)$ need be $\tau_{H_d}$-continuous on
 $\mathscr {P}_0(X)$ for a continuous function $f$.

 \begin{ex}
 For each $n \in \mathbb{N}$, let $A_n = \left\{\left(\frac{1}{n},\frac{k}{n}
 \right): k \in \mathbb{N}\right\}$, let $A = \{0\} \times [0,\infty)$ and let
 $X = A \cup \bigcup_{n \in \mathbb N} A_n$ equipped with the Euclidean metric
 $d$ for the plane. Define $f : X \rightarrow \mathbb{R}$ by
 \[
 f(x) = \begin {cases}  1,  & \textrm{if} \ x = \left(\frac{1}{n},\frac{k}{n}
 \right) \
 \textrm{with} \ k \geq n^2;\\
                          0, &  \textrm{otherwise}.
 \end {cases}
 \]
 As $f$ is zero on a neighborhood of each point of $A$ and all other points of
 $X$ are isolated, $f$ is continuous on $X$.  Clearly,
 \[
 \lim_{n\rightarrow\infty} \ H_d(A,A_n) = 0, \ \ \Omega (f,A_n) = 0,
 \]
 because each $A_n$ is a UC-subset, while $\Omega (f,A) = 1$.  Thus, $\Omega
 (f,\cdot)$ fails to be continuous at $A$ with respect to the $H_d$-pseudometric
 topology on $\mathscr {P}_0(X)$.
 \end{ex}

 The last example notwithstanding, given a metrizable space $X$ and a continuous
 real-valued function $f$ on it, we can always find a compatible metric
 $\widehat{d}$ for which $A \mapsto \Omega (f,A)$ is $H_{\widehat{d}}$-continuous
 on $\mathscr {P}_0(X)$, in fact, identically equal to zero: let $d$ be any
 compatible metric and put
 \[
 \widehat{d}(x,w) = d(x,w) + |f(x) - f(w)|,
 \]
 so that $f$ is globally uniformly continuous (in fact Lipschitz) with respect
 to $\widehat{d}$.  Still, one might look for stronger topologies on $\mathscr
 {P}_0(X)$ for a metrizable space $X$ for which $A \mapsto \Omega (f,A)$ is
 continuous for all compatible metrics on $X$ and for all continuous $f$ with
 values in an arbitrary metric target space.  We display such a topology here
 and use it to give a bona fide joint continuity of oscillation result. Finally,
 we show that a subset of $X$ is a UC-subset with respect to a particular
 compatible metric on $X$ as soon as $\Omega (f,A)$ is finite for all continuous
 real-valued functions $f$ on $X$.

 \section{Preliminaries} \label{Sec: Prelim}

All topological spaces will be assumed to contain at least two points.  If $X$
and $Y$ are topological spaces, we write $Y^X$ for the set of all functions
from $X$ to $Y$, and we denote the continuous functions from $X$ to $Y$ by
$C(X,Y)$.  We call an extended real-valued function defined on a topological
space $X$ \textit{upper semicontinuous} (resp. \textit{lower semicontinuous})
at $x \in X$ provided whenever $\langle x_\lambda \rangle_{\lambda \in \Lambda}$
is a net in $X$ convergent to $x \in X$, we have $\limsup_{\lambda \in \Lambda}
\ f(x_\lambda) \leq f(x)$ (resp. $\liminf_{\lambda \in \Lambda} \ f(x_\lambda)
\geq f(x)$).  Global upper semicontinuity means $\forall \alpha \in \mathbb{R},
\ f^{-1} ([-\infty,\alpha))$ is open, whereas global lower semicontinuity means
$\forall \alpha \in \mathbb{R}, \ f^{-1} ((\alpha,\infty])$ is open.  An extended
real-valued function on a topological space $X$ is continuous with respect to
the usual topology on the target space $[-\infty,\infty]$ if and only if it is
both lower semicontinuous and upper semicontinuous.

Let $\langle X,d \rangle$ be a metric space.  For $A \in \mathscr {P}_0(X)$ and
$x \in X$, we write $d(x,A)$ for $\inf\{d(x,a) : a \in A\}$, and
$\textrm{diam}_d(A)$ for
\[
\sup \{d(a_1,a_2) : a_1 \in A  \ \textrm{and} \ a_2 \in A\}.
\]

We now discuss some basic topologies on $\mathscr {P}_0(X)$ for a metrizable
topological space $X$.  In the literature, these topologies, called
\textit{hyperspace topologies}, are often restricted to the nonempty closed
subsets of $X$.  We restrict our attention to certain \textit{admissible}
hyperspace topologies, i.e., those for which $x \mapsto \{x\}$ is a topological
embedding \cite [p. 1] {Be}.   First, if $A$ and $B$ are nonempty subsets of
$X$, the \textit{Hausdorff distance} between them as determined by a compatible
metric $d$ is defined by
\[
H_d(A,B) := \inf \ \{\varepsilon > 0: A \subseteq S_d(B,\varepsilon) \
\textrm{and} \ B \subseteq S_d(A,\varepsilon)\}.
\]
Clearly,
\[
H_d(A,B) = H_d(\textrm{cl} (A), \textrm{\textrm{cl}}(B)),
\]
and if $d$
is unbounded, then we can find nonempty subsets $A$ and $B$ with $H_d(A,B) =
\infty.$  Hausdorff distance so defined gives an extended pseudometric on
$\mathscr {P}_0(X)$.  A countable local base for the topology $\tau_{H_d}$ that
it determines at $A \in \mathscr {P}_0(X)$ consists of all sets of the form
\[
\left\{B \in \mathscr {P}_0(X) : H_d(A,B) < \frac{1}{n}\right\},
\]
where $n$ runs over the positive integers.  It can be shown that $H_d(A,B)$ is
the uniform distance between the associated distance functionals $d(\cdot,A)$
and $d(\cdot,B)$ \cite [Theorem 1.5.1] {Be}, and that two compatible metrics
determined the same hyperspace topologies if and only if they are uniformly
equivalent \cite [Theorem 3.3.2] {Be}.

We next introduce two ``hit-and-miss" topologies on $\mathscr {P}_0(X)$, for
which we need some additional (now standard) notation. For $E \in \mathscr
{P}_0(X)$, we put
\[
E^+ := \{A \in \mathscr {P}_0(X) : A \subseteq E \},
\]
and for $\mathscr {E} \subseteq \mathscr {P}_0(X)$, we put
\[
\mathscr {E}^- := \{A \in \mathscr {P}_0(X) : \forall E \in \mathscr {E},
\ E \cap A \neq \emptyset \}.
\]
The \textit{finite topology} $\tau_{fin}$ on $\mathscr {P}_0(X)$, often called
the \textit{Vietoris topology}, is generated by all sets of the form $V^+$ where
$V$ is an open subset of $X$ plus all sets of the form $\mathscr{V}^-$ where
$\mathscr {V}$ is a finite family of open subsets of $X$ (see, e.g., \cite{Be,
KT, Mi}).  Replacing finite families of open sets by the larger collection
of locally finite families of open sets, we obtain the finer \textit{locally
finite topology} $\tau_{loc fin}$ \cite {Be,BHPV,BN,Ma}.  For nets of nonempty
closed subsets, we cite the following classical results: $\langle A_\lambda
\rangle$ converges in the finite (resp. locally finite) topology to $A$ if and
only if $\langle d(\cdot,A_\lambda) \rangle$ converges pointwise (resp. uniformly)
to $d(\cdot, A)$ for each metric $d$ compatible with the topology of $X$ \cite
{Be,BHPV,BLLN}.  As noted above, uniform convergence of distance functionals
with respect to a particular metric $d$ means $H_d$-convergence of the
underlying net of subsets; pointwise converge of distance functionals with
respect to $d$ is called $d$-\textit{Wijsman convergence} for the underlying
net of subsets (see, e.g., \cite {Be1,CJM,Co,LL,Zs}).

\section{Continuity of Oscillation with respect to $\tau_{locfin}$}
\label{Sec: Contin}

We first give an alternate presentation of the oscillation of a function $f$
between metric spaces at a nonempty subset that we will use in the sequel.

\begin{prop}  \label {alt}
Let $\langle X,d \rangle$ and $\langle Y,\rho \rangle$ be metric spaces and
let $f \in Y^X$. Then for each $A \in \mathscr {P}_0(X)$, we have
\[
\Omega (f,A) = \inf_{n \in \mathbb{N}} \  \sup_{a \in A}\ \omega_n (f, a).
\]
\end{prop}

\begin{proof}
Put
\[
\Omega_n^*(f,A) := \sup_{a \in A} \ \ \omega_n(f, a)
\]
and then put
\[
\Omega^*(f,A) := \inf_{n \in \mathbb{N}} \ \Omega_n^*(f,A).
\]
We must show that
\begin{itemize}
\item[(3.1)] $\Omega (f,A) \leq \Omega^*(f,A)$; and
\item[(3.2)] $\Omega^* (f,A) \leq \Omega (f,A)$.
\end{itemize}
In (3.1) we may assume that $\Omega^*(f,A)$ is finite, and in (3.2), we may assume
that $\Omega (f,A)$ is finite. For (3.1), suppose $\Omega^*(f,A) < \alpha < \infty$
and choose $n \in \mathbb{N}$  such that $\Omega_n^*(f,A) < \alpha$.  We claim
that $\Omega_{2n}(f,A) < \alpha$.  Let $x$ and $w$ be arbitrary members of
$S_d \left(A, \frac{1}{2n}\right)$ with $d(x,w) < \frac{1}{2n}$.  Choosing $a \in A$
with $d(x,a) < \frac{1}{2n}$,  we have $\{x,w\} \subseteq S_d\left(a,\frac{1}{n}\right)$,
and so
\[
\rho(f(x),f(w)) \leq \textrm{diam}_\rho \ f\left(S_d\left(a,\frac{1}{n}\right)
\right) \leq \Omega_n^*(f,A),
\]
so that
\[
\Omega_{2n}(f,A) \leq \Omega_n^*(f,A) < \alpha,
\]
which establishes the claim. This yields that $\Omega (f,A) \leq \Omega^*(f,A)$.

For (3.2), let $\alpha$ satisfy $\Omega(f,A) < \alpha < \infty$ and then choose
$n \in \mathbb{N}$  such that $\Omega_n (f,A) < \alpha$.   Let $a \in A$ be
arbitrary and choose $x,w$ in $S_d\left(a, \frac{1}{2n}\right)$. The triangle inequality
gives
\[
d(x,w) < 2 \cdot \frac{1}{2n} = \frac{1}{n},
\]
and since $\{x,w\} \subseteq S_d \left(A,\frac{1}{n}\right),$ we have
$\rho(f(x),f(w)) \leq \Omega_n(f,A)$.  This yields $\omega_{2n}(f,a) \leq
\Omega_n (f,A)$ and so
\[
\Omega^* (f,A)  \leq \Omega_{2n}^* (f,A) \leq \Omega_n (f,A) < \alpha,
\]
from which $\Omega^* (f,A) \leq \Omega (f,A)$ follows.
\end{proof}

As an application of our last result, we now provide a counterpart of (1.3) of Section
\ref{sec: Intro} for the sequence $\langle \Omega_n(f,\cdot) \rangle$.

\begin{prop}
 Let $\langle X,d \rangle$ and $\langle Y,\rho \rangle$ be metric spaces, and $f$ be a
 function from $X$ to $Y$. Then $f$ is globally uniformly continuous on $X$ if and only if $\langle
 \Omega_n (f,\cdot) \rangle$ converges uniformly to the zero function on ${\mathscr P}_0(X)$.
\end{prop}

\begin{proof}
Suppose that $f$ is globally uniformly continuous on $X$. Let $\varepsilon>0$ be arbitrary.
Then there exists $\delta >0$ such that for any $x, y\in X$ with $d(x,y) < \delta$,
$\rho(f(x),f(y)) < \varepsilon$. If we choose $n_0 \in \mathbb N$ such that $\frac{2}{n_0}
< \delta$, then we have $\omega_n(f,x) \leq \varepsilon$ for all $x\in X$ whenever $n \ge n_0$.
It follows from Proposition \ref{alt} that for any $A \in
{\mathscr P}_0(X)$,
\[
\Omega_{2n}(f, A) \le \Omega_n^*(f, A) \le \varepsilon
\]
whenever $n \ge n_0$. This means that  $\langle \Omega_n (f,\cdot) \rangle$ converges uniformly
to the zero function on ${\mathscr P}_0(X)$.

To see the converse, assume that $\langle \Omega_n (f,\cdot) \rangle$ converges uniformly
to the zero function on ${\mathscr P}_0(X)$. Let $\varepsilon>0$ be arbitrary. Then, by Proposition
\ref{alt}, one can find some $n_0 \in \mathbb N$ such that for any $A \in {\mathscr P}_0(X)$,
\[
\Omega_{2n}^*(f, A) \le \Omega_n(f, A) < \varepsilon
\]
whenever $n \ge n_0$. Thus, for any $x \in X$, we have
\[
\omega_{2n}(f,x) = \Omega_{2n}^*(f,\{x\}) < \varepsilon,
\]
whenever $n \ge n_0$. It follows that for any $x, y \in X$ with $d(x,y) < \frac{1}{2n_0}$,
we have $\rho(f(x), f(y)) < \varepsilon$. This confirms the global uniform continuity of
$f$ on $X$.
\end{proof}

In \cite {Be1}, it is shown that if  $\langle X,d \rangle$ and $\langle Y,\rho
\rangle$ are metric spaces and $f \in Y^X$, then $A \mapsto \Omega (f,A)$ is
upper semicontinuous on $\mathscr {P}_0(X)$ if it is equipped with $\tau_{H_d}$.
But the argument provided shows that something else is true, namely:

\begin{prop}
Let $X$ be a metrizable space and let $\langle Y,\rho \rangle$ be a metric space.
Then for each metric $d$ on $X$ that is compatible with the topology and for each
$f \in Y^X$, the assignment $A \mapsto \Omega (f,A)$ computed using $d$ and $\rho$
is upper semicontinuous on $\mathscr {P}_0(X)$ equipped with $\tau_{fin}$.
\end{prop}

\begin{proof}
Suppose $A \in \mathscr {P}_0(X)$.  There is nothing to prove if $\Omega(f,A) = \infty$.
Otherwise, take $\alpha \in \mathbb{R}$ with $\Omega (f,A) < \alpha$ as computed with
respect to the compatible metric $d$.  Choose $n \in \mathbb{N}$ with $\Omega_n
(f,A) < \alpha$.  Let $B \in S_d \left(A,\frac{1}{2n}\right)^+$ be arbitrary; by
definition, $B \subseteq S_d\left(A,\frac{1}{2n}\right)$, so we have
\[
\Omega (f,B) \leq \Omega_{2n}(f,B) \leq \Omega_n(f,A) < \alpha.
\]
The result now follows because $S_d\left(A,\frac{1}{2n}\right)^+$ is a
$\tau_{fin}$-neighborhood of $A$.
\end{proof}

Unfortunately, continuity of oscillation need not hold with respect to the finite
topology for each compatible metric on the domain, even if the function $f$ is
real-valued and globally continuous.  We provide a general construction in our
next example.

\begin{ex}
Suppose that we have a metric space $\langle X,d \rangle$  that contains a nonempty
subset $A$ that fails to be a UC subset but that is nevertheless a countable union
of nonempty UC-subsets (for example, any unbounded dense-in-itself subset of
$\mathbb{R}$ equipped with the usual metric does the job).  Since the nonempty
UC-subsets form a bornology, we can assume that there is an increasing sequence
of nonempty UC-subsets $\langle A_n \rangle$ with union $A$.   Clearly,  $\langle
A_n \rangle$ is  $\tau_{fin}$-convergent to $A$, but taking $f \in C(X,\mathbb{R})$,
say, that fails to be strongly uniformly continuous at $A$ (see \cite [Corollary 5.3]
{BL1}), we have $\Omega (f,A_n) = 0$ for each $n$ while $\Omega (f,A) > 0$.
\end{ex}

\begin{thm}  \label {lf}
Let $X$ be a metrizable space and let $\langle Y,\rho \rangle$ be a metric space.
Then for each $f \in C(X,Y)$ and for each compatible metric $d$ for $X, \ A \mapsto
\Omega (f,A)$ computed with respect to $d$ and $\rho$ is continuous on $\mathscr
{P}_0(X)$ equipped with the locally finite topology.
\end{thm}

\begin{proof}
Fix a metric $d$ compatible with the topology of $X$ and let $f \in C(X,Y)$. We
have already seen in our last result that $A \mapsto \Omega (f,A)$ is upper
semicontinuous when $\mathscr {P}_0(X)$ is equipped with the weaker finite
topology, so it remains to show the lower semicontinuity with respect to the
locally finite topology.

Suppose $\Omega (f,A) > \alpha \in \mathbb{R}$.  If \ $\Omega (f,A) = 0$, then
$\mathscr {P}_0(X)$ is a neighborhood of $A$ on which $\Omega (f,\cdot)$ exceeds
$\alpha$.  Otherwise, without loss of generality, we may assume that $0 < \alpha <
\Omega (f,A)$.  We intend to produce a locally finite family $\mathscr {V}$ of open
subsets of $X$ such that $A \in \mathscr {V}^-$, and whenever $ B\in \mathscr {V}^-$,
we have $\Omega (f,B) > \alpha.$

Select $\beta \in (\alpha,\Omega (f,A))$ and set $n_1 = 2$.  By Proposition \ref {alt},
we can find $a_1 \in A$  with
\[
\textrm{diam}_\rho \ f\left(S_d\left(a_1, \frac{1}{n_1}\right)\right) > \beta.
\]
By continuity of $f$ at $a_1$, we can find an even integer $n_2 > n_1$ such that
\[
\textrm{diam}_\rho \ f\left(S_d\left(a_1, \frac{1}{n_2}\right)\right) < \beta.
\]
Again by Proposition \ref {alt}, choose $a_2 \in A$ with
\[
\textrm{diam}_\rho \ f\left(S_d\left(a_2, \frac{1}{n_2}\right)\right) > \beta.
\]
Clearly, $a_2 \neq a_1$, and if $x \in S_d\left(a_1, \frac{1}{n_1}\right)$ then
$S_d\left(x,\frac{2}{n_1}\right) \supseteq S_d\left(a_1,\frac{1}{n_1}\right)$, for
whenever $w \in S_d\left(a_1,\frac{1}{n_1}\right)$, we have
\[
d(w,x) \leq d(w,a_1) + d(a_1,x) < \frac{1}{n_1} + \frac{1}{n_2} < \frac{2}{n_1}.
\]

Suppose we have chosen $2 = n_1 < n_2 < \cdots < n_k$ all even and distinct points
$a_1,a_2,\ldots,a_k$ in $A$ such that
\begin{enumerate}
\item[(3.3)] $\textrm{diam}_\rho \ f\left(S_d\left(a_j,\frac{1}{n_j}\right)\right) >
\beta \ \textrm{for} \ j = 1,2,\ldots, k;$

\item[(3.4)] $\textrm{diam}_\rho \ f\left(S_d\left(a_j,\frac{1}{n_{j+1}}\right)\right)
< \beta \ \textrm{for} \ j = 1,2,\ldots, k-1;$

\item[(3.5)] for each $x \in S_d\left(a_j,\frac{1}{n_{j+1}}\right)$, $S_d\left(x,
\frac{2}{n_j}\right) \supseteq S_d\left(a_j,\frac{1}{n_j}\right) \ \textrm{for} \
j = 1,2,\ldots, k-1.$
\end{enumerate}
By continuity of $f$ at $a_k$, we can choose an even integer $n_{k+1} > n_k$ satisfying
\[
\textrm{diam}_\rho \ f\left(S_d\left(a_k,\frac{1}{n_{k+1}}\right)\right) < \beta.
\]
By Proposition \ref {alt}, choose $a_{k+1} \in A$ such that
\[
\textrm{diam}_\rho \ f(S_d(a_{k+1},\frac{1}{n_{k+1}})) > \beta.
\]
By condition (3.4), $a_1,a_2,\ldots, a_k,a_{k+1}$ are all distinct, and an easy calculation
shows that for each $x \in S_d\left(a_k,\frac{1}{n_{k+1}}\right),$ we have $S_d\left(x,
\frac{2}{n_k}\right) \supseteq S_d\left(a_k,\frac{1}{n_k}\right).$

Continuing to produce a strictly increasing sequence of even integers $\langle n_k
\rangle$ and a sequence of distinct points $\langle a_k \rangle$ in $A$, we conclude
that the family of balls
\[
\left\{S_d\left(a_k,\frac{1}{n_k}\right) : k \in \mathbb{N}\right\}
\]
is locally finite, else $\langle a_k \rangle$ would have a cluster point at which continuity
of $f$ must fail by condition (3.3).  Thus, $\{S_d\left(a_k,\frac{1}{n_{k+1}}\right) : k
\in \mathbb{N}\}$, being a family of smaller balls, is also locally finite and $A \in
\{S_d\left(a_k,\frac{1}{n_{k+1}}\right): k \in \mathbb{N}\}^-.$  We intend to show that
if $B \in \mathscr {P}_0(X)$ hits each ball $S_d\left(a_k,\frac{1}{n_{k+1}}\right)$, then
$\Omega(f,B)>\alpha$.  To see this, it suffices by Proposition \ref {alt} to produce for
each $m \in \mathbb{N}$ some $b \in B$ with
\[
\textrm{diam}_\rho f\left(S_d\left(b, \frac{1}{m}\right)\right) > \beta.
\]

Choose $n_j$ with $\frac{2}{n_j} < \frac{1}{m}$ and then $b \in B \cap S_d\left(a_j,
\frac{1}{n_{j+1}}\right)$. Using condition (3.3) and condition (3.5),
\[
f\left(S_d\left(b,\frac{1}{m}\right)\right) \supseteq f\left(S_d\left(b,\frac{2}{n_j}
\right)\right) \supseteq f\left(S_d\left(a_j,\frac{1}{n_j}\right)\right)
\]
so that by condition (3.3) and Proposition \ref {alt}, $\Omega (f, B) \geq \beta > \alpha.$
\end{proof}

\section{Joint Continuity of Oscillation}

In \cite {BL1}, Beer and Levi introduced the variational notion of strong uniform convergence
of a net of functions $\langle f_\lambda \rangle_{\lambda \in \Lambda}$ from $\langle X,d
\rangle$ to $\langle Y,\rho \rangle$ to a function $f$ on a nonempty subset $A$ of $X$: for
each $\varepsilon >0$ there exists $\lambda_0 \in \Lambda$ such that for each $\lambda \succeq
\lambda_0$, there exists $\delta > 0$ such that for all $x \in S_d(A,\delta), \ \rho (f_\lambda (x),
f(x)) < \varepsilon$ (notice that $\delta$ can depend on $\lambda$!). The family of nonempty
subsets on which strong uniform convergence occurs is stable under finite unions and is
hereditary, and strong uniform convergence on (each member of) a bornology is compatible with
a uniformizable topology on $Y^X$.   If each $f_\lambda \in C(X,Y)$  and $f \in C(X,Y)$ and
$\langle f_\lambda \rangle_{\lambda \in \Lambda}$ is pointwise convergent to $f$, then strong
uniform convergence must occur on each singleton subset of $X$.  Conversely, strong uniform
convergence of a net of continuous functions on each singleton to $f \in Y^X$ ensures that $f$
is continuous \cite {BL1,Bou}.  One cannot overstate how well strong uniform convergence on a
bornology $\mathscr {B}$ comports with strong uniform continuity of functions on $\mathscr {B}$.
A number of subsequent papers on this convergence notion/topology have been written
\cite {BL2,CT,CDH,Ho};  in particular, nontransparent necessary and sufficient conditions on a
bornology $\mathscr {B}$ for the topology of strong uniform convergence to collapse to the
classical topology of $\mathscr{B}$-uniform convergence on $C(X,Y)$ \cite {MN} are known
\cite {BL2}.

But one result that we would hope for is not be had:  joint upper semicontinuity of $(f,A)
\mapsto \Omega (f,A)$ need not hold  even if we restrict our functions to those that are
strongly uniformly continuous on a bornology $\mathscr {B}$ and restrict our sets to members
of $\mathscr {B}$, where our functions are topologized by the topology of strong uniform
convergence on the bornology and $\mathscr {B}$ is topologized by Hausdorff distance \cite
[Example 6.13] {BL1}.   What is needed is a somewhat stronger convergence notion for our
functions that is obtained by flipping quantifiers in the definition of strong uniform
convergence.

\begin{defn}
Let $\langle X,d \rangle$ and $\langle Y,\rho \rangle$ be metric spaces, and let  $A$ be a
nonempty subset of $X$.  A net of functions $\langle f_\lambda \rangle_{\lambda \in \Lambda}$
from $X$ to $Y$ is declared \textit{very strongly uniformly convergent to} $f:X \rightarrow Y$
\textit{on} $A$ if for each $\varepsilon >0$ there exists $\lambda_0 \in \Lambda$ and
$\delta > 0$ such that for all $\lambda \succeq \lambda_0$ and all $x \in S_d(A,\delta)$,
$\rho (f_\lambda (x), f(x)) < \varepsilon$.
\end{defn}

To see the difference between strong uniform convergence and very strong uniform convergence
for continuous functions, let $f_n$ be the piecewise linear function on $[0,1]$ whose graph
joins $(0,0)$ to $(\frac{1}{2n},1)$ to $(\frac{1}{n},0)$ to $(1,0)$, let $f$ be the zero
function on $[0,1]$, and let $A = \{0\}$. Then, it is readily checked that
$\langle f_n \rangle$ is strongly uniformly convergent to $f$ on $A$, but $\langle f_n
\rangle$ is not very strongly uniformly convergent to $f$ on $A$.

While strong uniform convergence on each singleton of a pointwise convergent net of continuous
functions, being equivalent to continuity of the limit,  cannot ensure uniform convergence on
compact subsets, we have the following result whose easy proof is left to the reader.

\begin{prop}
Let $\langle X,d \rangle$ and $\langle Y, \rho \rangle$ be metric spaces and let $\langle
f_\lambda \rangle_{\lambda \in \Lambda}$ be a net in $Y^X$ very strongly uniformly convergent
to $f: X \rightarrow Y$ on each singleton subset of $X$.  Then very strong uniform convergence
on compact subsets to $f$ occurs.
\end{prop}

\begin{thm}
Let $\langle X,d \rangle$ and $\langle Y, \rho \rangle$ be two metric spaces.  Let $\langle
f_\lambda \rangle_{\lambda \in \Lambda}$ be a net in $Y^X$ and $\langle A_\lambda
\rangle_{\lambda \in \Lambda}$ be a net of nonempty subsets of $X$, where $\Lambda$ is a
directed set.  Suppose that $A \in \mathscr{P}_0(X)$ and $f : X \rightarrow Y$ satisfy
\begin{enumerate}
\item[{\rm (i)}] for all $n \in \mathbb{N}, A_\lambda \subseteq S_d\left(A,\frac{1}{n}\right)$
eventually; and
\item[{\rm (ii)}] $\langle f_\lambda \rangle_{\lambda \in \Lambda}$ is very strongly
uniformly convergent to $f \in Y^X$ on $A$.
\end{enumerate}
Then
\[
\Omega (f,A) \geq \limsup_{\lambda \in \Lambda} \ \Omega (f_\lambda,A_\lambda).
\]
\end{thm}

\begin{proof}
We may assume that $\Omega (f,A)$ is finite.  Let $\beta > \Omega (f,A)$ be arbitrary and
choose $\varepsilon > 0$ such that $\beta > \Omega (f,A) + 3\varepsilon.$ Choose $n \in
\mathbb{N}$ and $\lambda_0 \in \Lambda$ so large that
\[
\Omega_n (f,A) < \Omega (f,A) + \varepsilon
\]
and such that whenever $\lambda \succeq \lambda_0$ we have both
\begin {enumerate}
\item[(4.1)] $\sup \left\{\rho(f(x),f_\lambda(x)) : x \in S_d\left(A,\frac{1}{n}\right)
\right\} < \varepsilon$;
\item[(4.2)] $A_\lambda \subseteq S_d\left(A,\frac{1}{2n}\right)$.
\end {enumerate}

For $\lambda \succeq \lambda_0$, by (4.2) we have $S_d\left(A_\lambda,\frac{1}{2n}\right) \subseteq
S_d\left(A, \frac{1}{n}\right)$, and so whenever $\{x,w\} \subseteq S_d\left(A_\lambda,
\frac{1}{2n}\right)$ and $d(x,w) < \frac{1}{2n}$, we get
\begin{eqnarray*}
\rho(f_\lambda(x), f_\lambda(w)) &\leq& \rho(f_\lambda(x), f(x)) + \rho (f(x),f(w)) + \rho
(f(w),f_\lambda(w))\\
&<& \varepsilon + \Omega_n (f,A) + \varepsilon < \Omega (f,A) + 3\varepsilon.
\end{eqnarray*}
From this, it follows that whenever $\lambda \succeq \lambda_0$,
\[
\Omega (f_\lambda,A_\lambda) \leq \Omega_{2n} (f_\lambda,A_\lambda) \leq \Omega (f,A) +
3\varepsilon < \beta
\]
so that
\[
\Omega (f,A) \geq \textrm{lim sup}_{\lambda \in \Lambda} \ \Omega (f_\lambda,A_\lambda),
\]
and the proof is complete.
\end{proof}

\begin{thm}
Let $X$ be a metrizable space and $\langle Y, \rho \rangle$ be a metric space. Let $(\Lambda,
\succeq)$ be a directed set and let $\langle f_\lambda \rangle_{\lambda \in \Lambda}$ be a net
in $C(X,Y)$ and let $\langle A_\lambda \rangle_{\lambda \in \Lambda}$ be a net in ${\mathscr
P}_0(X)$. Suppose that $\langle A_\lambda \rangle_{\lambda \in \Lambda}$ is convergent in the
locally finite topology $\tau_{locfin}$ to $A \in \mathscr{P}_0(X)$, and suppose $\langle
f_\lambda \rangle_{\lambda \in \Lambda}$ is very strongly uniformly convergent to $f \in
C(X,Y)$ on $A$. Then for each metric $d$ compatible with the topology of $X$,
\[
\Omega (f,A) = \lim_{\lambda \in \Lambda} \ \Omega (f_\lambda,A_\lambda),
\]
where oscillations are computed with respect to $d$ and $\rho$.
\end{thm}

\begin{proof}
Since containment in $S_d(A,\delta)$ means membership to $S_d(A,\delta)^+$ and $S_d(A,\delta)^+$
belongs to the locally finite topology $\tau_{locfin}$, the last result gives
\[
\textrm{lim sup}_{\lambda \in \Lambda} \ \Omega (f_\lambda,A_\lambda) \leq \Omega (f,A).
\]  It remains to show that
\[
\Omega (f,A) \leq \textrm{lim inf}_{\lambda \in \Lambda} \ \Omega (f_\lambda,A_\lambda).
\]

We may assume that $\Omega (f,A) \neq 0$.  Let $0 < \alpha < \Omega (f,A)$  be arbitrary, and
choose $\varepsilon > 0$ with $\alpha + 3\varepsilon < \Omega (f,A)$.  The proof of Theorem
\ref {lf} shows that there exists a locally finite family $\mathscr {V}$ of nonempty open subsets
of $X$ (in fact open balls) such that $A \in \mathscr {V}^-$ and if $B \in \mathscr {V}^-$, then
$\Omega (f,B) > \alpha + 3\varepsilon.$  By very strong uniform convergence on $A$, choose
$\delta > 0$ and $\lambda_1 \in \Lambda$ such that whenever $\lambda \succeq \lambda_1$,
\[
\textrm{sup} \ \{\rho(f(x),f_\lambda(x)) : x \in S_d(A,\delta) \} < \varepsilon.
\]
There exists $\lambda_2 \succeq \lambda_1$ such that
\[
A_\lambda \in \mathscr {V}^- \cap S_d\left(A, \frac{\delta}{2}\right)^+
\]
for all $\lambda \succeq \lambda_2$.
Let $n$ be any positive integer such that $\frac{1}{n} <\frac{\delta}{2}$.  Whenever $\{x,w\}
\subseteq S_d\left(A_\lambda,\frac{1}{n}\right)$ and $d(x,w) < \frac{1}{n}$, we have  $\{x,w\}
\subseteq S_d(A,\delta)$, and the triangle inequality yields for all $\lambda \succeq
\lambda_2$,
\[
\rho (f_\lambda(x),f_\lambda(w)) > \rho(f(x),f(w)) - 2\varepsilon.
\]
It follows that for all $n$ sufficiently large, we get
\[
\Omega_n(f_\lambda,A_\lambda) \geq \Omega_n(f,A_\lambda) - 2\varepsilon  \geq
\Omega(f,A_\lambda) - 2\varepsilon > \alpha + \varepsilon,
\]
and so
\[\Omega(f_\lambda,A_\lambda) \geq \alpha + \varepsilon > \alpha
\]
for all $\lambda \succeq \lambda_2$.  We may now conclude that
\[
\Omega (f,A) \leq \noindent \liminf_{\lambda \in \Lambda} \ \Omega (f_\lambda,A_\lambda),
\]
which completes the proof.
\end{proof}

\begin{cor}
Let $X$ be a metrizable space and $\langle Y, \rho \rangle$ be a metric space. Suppose $C(X,Y)$ is
equipped with the topology of $\rho$-uniform convergence, and $\mathscr {P}_0(X)$ is equipped with
the locally finite topology. Then for each metric $d$ compatible with the topology of $X, \ (f,A)
\rightarrow \Omega (f,A)$   is continuous on $C(X,Y) \times \mathscr {P}_0(X)$ equipped with the
product topology, where oscillations are computed with respect to $d$ and $\rho$.
\end {cor} \vspace {2 mm}

If we have very strong uniform convergence of a net of functions on each member of some family of
nonempty subsets $\mathscr {A}$ of $\langle X,d \rangle$ then we evidently have very strong uniform
convergence on subsets of members of  $\mathscr {A}$ and on finite unions of members of $\mathscr {A}$.
Thus, if $\mathscr {A}$ forms a cover of $X$, there is no loss of generality in assuming that $\mathscr{A}$
is a bornology.

Let $\langle X,d \rangle$ and $\langle Y,\rho \rangle$ be metric spaces, and let $\mathscr {B}$ be a
bornology  on $X$. We say that a net of functions $\langle f_\lambda \rangle_{\lambda \in \Lambda}$
from $X$ to $Y$ is \textit{very strongly uniformly convergent to} $f:X \rightarrow Y$ \textit{on}
$\mathscr {B}$ if it is very strongly uniformly convergent to $f$ on every member $B$ of $\mathscr B$.
This usage parallels the notion of strong uniform convergence in $Y^X$ on $\mathscr {B}$ where the
convergence is always compatible with a uniformizable topology \cite {BL1}.  It came as a surprise
to the authors that very strong uniform convergence need not be topological even if we restrict our
attention to $C(X,\mathbb{R})$ and our bornology is the bornology of nonempty finite subsets $\mathscr
{F}_0(X)$, which corresponds to very strong uniform convergence on each singleton subset.  We show
that the iterated limit condition \cite [p. 30] {KT} which is necessary for the convergence to be
topological can fail for a sequence of sequences of real-valued continuous functions. \vspace {2 mm}

\begin {ex}  Our base metric space $\langle X,d \rangle$ is the sequence space $\ell_\infty$, in
which for each $n \in \mathbb{N}, \ e_n$ is the sequence whose $n$th term is one and whose other
terms are all zero.  For our $k$th sequence in $C(\ell_\infty,\mathbb{R})$ we put
\[
f_{k,n}(x) = \begin{cases}
                 1 - {3^k}d(x,\frac{1}{k}e_n), & \textrm{if} \ d(x,\frac{1}{k}e_n) < \frac{1}{3^k}; \\
                 0, & \textrm{otherwise}.
             \end {cases}
\]
The sequence $f_{k,1}, f_{k,2},\ldots, f_{k,n}, \ldots$ converges very strongly uniformly at
each point of $\ell_\infty$ to the zero function as it is eventually zero in a fixed
neighborhood of each point that does not depend on $n$.  However, if we direct $\mathbb{N}
\times \mathbb{N}^\mathbb{N}$ with the pointwise order, the net $(k,\phi) \mapsto f_{k,\phi(k)}$ evidently fails to converge very strongly uniformly at the origin of $\ell_\infty$ to the zero
function.
\end {ex}

\section {A New Characteriation of UC Subsets}

In the introduction, we described sequentially those nonempty subsets $A$ of a metric space $\langle X,d
\rangle$ on which each continuous function on $X$ has oscillation zero at $A$; for a very different
sequential description, the reader may consult \cite [Theorem 3.5] {BDM}.  Such subsets, called UC-subsets,
can also be described in terms of gaps, where the \textit{gap} between two nonempty subsets of $X$ is the
infimum of the distances between pairs of points one in each set, see \cite {Be}. We call two nonintersecting
sets \textit{asymptotic} if the gap between them is zero.

For notational economy and following \cite {BL1}, we now write $I(x)$ for $d(x,X \backslash \{x\})$; the
functional $I(\cdot)$ measures the isolation of points of $X$ and $I(x) = 0$ means that $x$ is a limit
point of $X$.

\begin{thm}
Let $A$ be a nonempty subset of $\langle X,d \rangle$.  The following conditions are equivalent:
\begin{enumerate}
\item[{\rm (i)}]$A$ is a UC-subset;

\item[{\rm (ii)}]  for each $f \in C(X,\mathbb{R}), \ \Omega (f,A)$ is finite;

\item[{\rm (iii)}] whenever $C$ and $E$ are nonempty closed subsets of $X$ with $C\subseteq
A$ and $C \cap E = \emptyset$, then the gap between $C$ and $E$ is positive.
\end{enumerate}
\end{thm}

\begin{proof}  (i) $\Rightarrow$ (ii) trivially follows from \cite [Theorem 5.2] {BL1}.

For (ii) $\Rightarrow$ (iii), if (iii) fails for some $C$ and $E$, by the Tietze extension theorem, we can
find $f \in C(X,\mathbb{R})$ mapping each point of $C$ to zero and such that for each $e \in E, f(e) =
d(e,C)^{-1}$.   Now for each $n \in \mathbb{N}$, there exists $c_n \in C$ and $e_n \in E$ with $d(c_n,e_n)
< \frac{1}{n}$ and of course  $\{c_n,e_n\} \subseteq S_d\left(A,\frac{1}{n}\right)$.  As a result
$\Omega_n(f,A) > n$ and so $\Omega (f,A) = \infty$.

Only (iii) $\Rightarrow$ (i) remains. Suppose (i) fails; then we can find a sequence $\langle a_n \rangle$
in $A$ that fails to cluster but for which $\textrm{lim}_{n\rightarrow\infty} \ I(a_n) = 0$.  By passing
to a subsequence, we may assume all terms are distinct and either (a) all terms are limit points of $X$,
or (b) all terms are isolated points of $X$.  In case (a), choose a strictly increasing
sequence of positive integers $\langle k_n\rangle$ such that for each $n \in \mathbb{N}$,
\[
\frac{1}{k_n} < \frac{1}{3}\ d(a_n, \{a_j : j \neq n\});
\]
then $\left\{S_d\left(a_n,\frac{1}{k_n}\right): n \in \mathbb{N}\right\}$ is a pairwise disjoint
family of balls. Choose $e_n \in S_d\left(a_n,\frac{1}{k_n}\right)$ different from $a_n$, and put
$C := \{a_n : n \in \mathbb{N}\}$ and $E :=
\{e_n : n \in \mathbb{N}\}$. Since $\langle a_n \rangle$ can't cluster, neither can $\langle e_n \rangle$
as $\textrm{lim}_{n\rightarrow\infty} \ d(a_n, e_n) = 0$.  Thus, $C$ and $E$ are disjoint asymptotic
closed subsets of $X$ with $C \subseteq A$, which violates (iii).

Case (b) is a little more delicate, involving an iterative procedure.  Put $n_1 = 1$ and choose $x_{n_1}
\in X$ so that
\[
0 < d(a_{n_1},x_{n_1}) < I(a_{n_1}) + \frac{1}{n_1}.
\]
Having chosen $n_1 < n_2 < \cdots < n_k$ and $x_{n_1}, x_{n_2}, \ldots, x_{n_k}$ not necessarily distinct
in $X$ such that $\{a_{n_1}, a_{n_2}, \ldots, a_{n_k}\}$ and $\{x_{n_1}, x_{n_2}, \ldots, x_{n_k}\}$ form
disjoint sets, put
\[
\delta = \textrm{min} \ \{\{I(a_{n_j}) : j \leq k\}, \{I(x_{n_j}) : x_{n_j} \ \textrm{is an isolated point
of} \ X, j \le k \}\}.
\]
Since the measure of isolation functional goes to zero along the sequence $\langle a_n \rangle$, we can
find $n_{k+1} > n_k$ and then $x_{n_{k+1}} \in X$ with
\[
0 < d(a_{n_{k+1}},x_{n_{k+1}}) < I(a_{n_{k+1}}) + \frac{1}{n_{k +1}} < \delta.
\]
Since $I(x_{n_{k+1}}) < \delta$ as well, we obtain the disjointness of  $\{a_{n_1}, a_{n_2}, \ldots,
a_{n_{k+1}}\}$ and $\{x_{n_1}, x_{n_2}, \ldots, x_{n_{k+1}}\}$.

In this way we produce sequences $\langle a_{n_j} \rangle$ and $\langle x_{n_j} \rangle$ whose sets of
terms do not overlap and such that for each $j \in \mathbb{N}$,
\[
0 < d(a_{n_j},x_{n_j}) < I(a_{n_j}) + \frac{1}{n_j}.
\]
Since $\langle a_{n_j} \rangle$ can't cluster, neither can $\langle x_{n_j} \rangle$ as
$\textrm{lim}_{j\rightarrow\infty} \ d(a_{n_j},x_{n_j}) = 0$.  Thus, $C = \{a_{n_j} : j \in \mathbb{N}\}$
and $E = \{x_{n_j} : j \in \mathbb{N}\}$ are disjoint asymptotic closed subsets of $X$ with $C \subseteq
A$, once again violating condition (iii).
\end{proof}

\bigskip

\bibliographystyle{amsplain}

\end{document}